\documentclass[12pt]{amsart}
\usepackage{amssymb}

\theoremstyle{plain}
\newtheorem{theorem}{Theorem}[section]
\newtheorem{corollary}[theorem]{Corollary}
\newtheorem{lemma}[theorem]{Lemma}
\newtheorem{proposition}[theorem]{Proposition}

\theoremstyle{definition}
\newtheorem{definition}[theorem]{Definition}
\newtheorem{remark}[theorem]{Remark}
\newtheorem*{notation}{Notation}

\newtheorem{example}[theorem]{Example}

\newcommand{\abs}[1]{\lvert#1\rvert}
\newcommand{\norm}[1]{\lVert#1\rVert}

\newcommand{\bignorm}[1]{\bigl\lVert#1\bigr\rVert}

\renewcommand{\le}{\leqslant}
\renewcommand{\ge}{\geqslant}
\renewcommand{\mid}{\::\:}

\newcommand{\term}[1]{{\textit{\textbf{#1}}}}

\def\rk{{\rm rank}\,}
\def\ran{{\rm ran}\,}

\def\hilb{\mathcal H}
\def\bofh{\mathcal B\left(\mathcal H\right)}
\def\range{\mathrm{range}\,}

\def\bbC{\mathbb C}
\def\bbD{\mathbb D}

\def\bbN{\mathbb N}

\def\bbR{\mathbb R}

\def\cB{\mathcal B}

\def\cE{\mathcal E}
\def\cF{\mathcal F}

\def\cH{\mathcal H}

\def\cJ{\mathcal J}
\def\cK{\mathcal K}
\def\cL{\mathcal L}

\def\cN{\mathcal N}

\def\cP{\mathcal P}

\def\cS{\mathcal S}
\def\cT{\mathcal T}
\def\cU{\mathcal U}

\begin{document}
\baselineskip 18pt

\title[Semigroups of Partial Isometries]
{Semigroups of Partial Isometries}

\author[A.I.~Popov]{Alexey I. Popov${}^1$}
\address{Department of Pure Mathematics, University of Waterloo, Waterloo, ON, N2L\,3G1. Canada}
\email{a4popov@uwaterloo.ca}
\author[H.~Radjavi]{Heydar Radjavi${}^1$}
\address{Department of Pure Mathematics, University of Waterloo, Waterloo, ON, N2L\,3G1. Canada}
\email{hradjavi@uwaterloo.ca}
\thanks{${}^1$ Research supported in part by NSERC (Canada)}
\keywords{Partial isometry, semigroup, self-adjoint, spectrum, weighted composition}
\subjclass[2010]{47A65, 47D03, 47B33, 20M20}

\date{\today.}
\begin{abstract}
We study self-adjoint semigroups of partial isometries on a Hilbert space. These semigroups coincide precisely with faithful representations of abstract inverse semigroups. Groups of unitary operators are specialized examples of self-adjoint  semigroups of partial isometries. We obtain a general structure result showing that every self-adjoint semigroup of partial isometries consists of ``generalized weighted composition'' operators on a space of square-integrable Hilbert-space valued functions. If the semigroup is irreducible and contains a compact operator then the underlying measure space is purely atomic, so that the semigroup is represented as ``zero-unitary'' matrices. In this case it is not even required that the semigroup be self-adjoint.
\end{abstract}

\maketitle


\section{Introduction}

It has been known for some time (see~\cite{DP85}; see also \cite[Proposition~2.1.4]{Pat98}) that abstract inverse semigroups are faithfully representable as self-adjoint semigroups of partial isometries on a Hilbert space. We propose to study the spacial structure of these representations.

This paper is mainly concerned with operator semigroups consisting of partial isometries on a separable Hilbert space. A \emph{partial isometry} on a Hilbert space $\hilb$ is an operator $A$ for which $A^*A$ and $AA^*$ are both self-adjoint projections. By a \emph{semigroup} of operators on $\hilb$ we simply mean a set $\cS$ closed under multiplication; it is said to be self-adjoint if $\cS = \cS^*:= \{A^*: A\in\cS\}$. Thus the concept of self-adjoint semigroups of partial isometries is a direct and natural  generalization of that of groups of unitary operators.  It is also a direct generalization of the concept of a self-adjoint \emph{band}, or a semigroup of self-adjoint projections (which is automatically abelian), on~$\hilb$.  Not surprisingly,  our  results will demonstrate how the two concepts of unitary group and band intertwine to yield the desired structure. For the simplest example in finite dimensions, just consider the semigroup $\cS$ consisting of basic $r\times r$ matrices $E_{ij}$  together with the zero matrix. (The basic matrix $E_{ij}$ has only one nonzero entry at $(i, j)$ which is~$1$.) Also,  let $\cU$ be a group of $s\times s$ unitary matrices.  Then  the tensor product  $\cS \otimes \cU$ is a self-adjoint semigroup of partial isometrics. It contains the band $\cB = \{E_{ii}\} \otimes I_s$  which partitions the underlying space into $s$-dimensional subspaces (i.e., the ranges of  projections in~$\cB$) which are isometrically permuted via the unitary group~$\cU$.

Halmos and Wallen~\cite{HW69} gave a thorough description of  the structure of \emph{power partial isometries}, i.e., those $A\in\bofh$ for which $A^n$ is a partial isometry for every natural number~$n$. As will be seen below, this constitutes the special case of  finitely generated  situation in our study of self-adjoint semigroups of partial isometries.
      
Throughout the paper, the symbol $\hilb$ will stand for a complex separable Hilbert space, perhaps finite-dimensional. The term \emph{operator} will refer to a bounded linear map on a Banach space (usually,~$\hilb$). If $X$ and $Y$ are (closed) subspaces of~$\hilb$ then $X\vee Y$ will denote, as usual, the closure of the span of $X\cup Y$. 

A semigroup of operators is called \emph{irreducible} if it has no common (closed, non-trivial) invariant subspaces. It is easy to see that a semigroup of operators is irreducible if and only if the algebra generated by $\cS$ has no common invariant subspaces. 

The operator semigroup generated by a set $\cF$ of operators will be denoted by $\cS(\cF)$. If  $\cF=\{T_1,\dots,T_n\}$ then $\cS(\cF)$ will also be denoted by $\cS(T_1,\dots,T_n)$. In particular, the self-adjoint semigroup generated by $T\in\bofh$ is denoted by $\cS(T,T^*)$. Similarly, if $\cF=\cF_1\cup\cF_2$, we will sometimes use the symbol $\cS(\cF_1,\cF_2)$ to denote the semigroup $\cS(\cF)$. We remark that the semigroup $\cS(\cF)$ consists of all the possible words of the form $A_1A_2\cdots A_k$, where $k\in\bbN$ and for each $i\in\{1,\dots,k\}$, the operator $A_i\in\cF$.

For general texts about the matrix and operator semigroups, we refer the reader to monographs~\cite{Okninski} and \cite{RR00}.


\section{Enriching semigroups with projections}\label{section-2}

In this section we prove that the set of projections in a self-adjoint semigroup of partial isometries can be enlarged to a complete Boolean algebra in such a way that the resulting semigroup still consists of partial isometries. This result, which will be used in Section~\ref{section-3}, seems to be of independent interest.

If an operator $T$ belongs to a semigroup of partial isometries, then it must satisfy the condition that $T$ and all of its powers are partial isometries. Such operators are called \term{power partial isometries}. The following result characterizing power partial isometries was obtained by Halmos and Wallen in~\cite{HW69}. Recall that a pure isometry is just a unilateral shift of arbitrary multiplicity, a pure co-isometry is the adjoint of a pure isometry, and a truncated shift is a Jordan block with zero diagonal.

\begin{theorem}\label{pwr-prt-isometry}\cite{HW69}
Every power partial isometry decomposes as a direct sum whose summands are  unitary operators, pure isometries, pure co-isometries, and truncated shifts.
\end{theorem}

The starting point for our study is the following simple corollary of the result of Halmos and Wallen. It can be viewed as an alternative characterization of power partial isometries in terms of singly generated self-adjoint semigroups.

\begin{proposition}\label{semigroup-ppi}
Let $T$ be an operator in $\bofh$. Then $T$ is a power partial isometry if and only if the self-adjoint semigroup $\cS(T,T^*)$ generated by $T$ consists of partial isometries.
\end{proposition}
\begin{proof}
If $\cS(T,T^*)$ consists of partial isometries, then, obviously, $T$ is a power partial isometry. On the other hand, if $T$ is a power partial isometry, then, by Theorem~\ref{pwr-prt-isometry}, it is a direct sum of a unitary, forward and backward shifts, and finite-dimensional truncations of the shift. It is easy to check that each such summand generates a self-adjoint semigroup consisting of partial isometries. Hence, so does the operator $T$ itself.
\end{proof}

Let us now turn to the study of projections in self-adjoint semigroups of partial isometries. The set of projections in an operator semigroup $\cS$ will be denoted by $\cP(\cS)$.

\begin{proposition}\label{proj-band}
Let $\cS$ be a self-adjoint semigroup of partial isometries. Then every idempotent in $\cS$ is self-adjoint and the set $\cP(\cS)$ of projections in $\cS$ is a commutative subsemigroup of~$\cS$.
\end{proposition}
\begin{proof}
Let $E=E^2\in\cS$ be an idempotent. With respect to the decomposition $\hilb=\range E\oplus(\range E)^\perp$, write
$$
E=
\begin{bmatrix}
I & X\\
0 & 0
\end{bmatrix}.
$$
Since $E$ is a partial isometry, it must be a contraction, so that $X=0$.

The claim that all projections in $\cS$ commute follows, for instance, from \cite[Lemma~2]{HW69}, which states that the product $UV$ of two partial isometries $U$ and $V$ is a partial isometry if and only if the projections $U^*U$ and $VV^*$ commute. Finally, the product of two commuting projections is necessarily a projection.
\end{proof}

Clearly, every self-adjoint semigroup $\cS$ of partial isometries contains some projections: if $T\in\cS$ is any operator in $\cS$ then $T^*T$ and $TT^*$ are projections in~$\cS$. Our next step is to show that one can always enlarge the semigroup in such a way that it becomes \emph{rich} with projections.

Recall that a semigroup of idempotent operators is called a \term{band} (see \cite[Definition 2.3.4]{RR00}). The fact that the projections in $\cS$ commute allows us to introduce the following object. 

\begin{definition}
Let $\cS=\cS^*$ be a semigroup of partial isometries and $\cP(\cS)$ be the set of projections in~$\cS$. The \term{enveloping band of projections} of $\cS$ is the complete Boolean algebra $\cE(\cS)$ of projections in $\bofh$ generated by $\cP(\cS)$. 
\end{definition}

We note that the enveloping band of projections of $\cS$ is necessarily a commutative band.

\begin{theorem}\label{env-band}
Suppose that $\cS_0$ is a self-adjoint semigroup of partial isometries. Then the semigroup $\cS_1=\cS\big(\cS_0,\cE(\cS_0)\big)$ generated by $\cS_0$ and the enveloping band of projections of $\cS_0$ is a self-adjoint semigroup of partial isometries such that $\cP(\cS_1)=\cE(\cS_1)=\cE(\cS_0)$.
\end{theorem}
\begin{proof}
For the sake of clarity, let us split the proof into several simple claims.

{\it Claim 1}. If $A\in\cS_0$ and $E\in\cP(\cS_0)$ then there exists $F\in\cP(\cS_0)$ such that $EA=AF$. 

To prove the claim, put $F=A^*EA$. Clearly, $F\in\cS_0$. Also,
$$
F=A^*EA=(A^*E)(EA)=(EA)^*(EA)
$$
is a projection, since $EA\in\cS$ is a partial isometry. Finally, letting $P=AA^*$ and using $PE=EP$, we obtain
$$
AF=AA^*EA=PEA=EPA=EA.
$$
This proves {\it Claim 1}.

{\it Claim 2}. Let $A\in\cS_0$ and $E\in\cE(\cS_0)$. Then there exists $F\in\cE(\cS_0)$ such that $EA=AF$. 

Let us denote the Boolean algebra generated by $\cP(\cS_0)$ by~$\cB$. Then $\cE(\cS_0)$ consists of all the suprema and infima of subsets of projections in~$\cB$. We will first establish that the claim holds for all projections in~$\cB$.

Let $E\in\cB$ be arbitrary. Then $E$ is obtained as a finite word consisting of projections in $\cP(\cS_0)$ and links 
$$
a\wedge b = ab,\quad \neg a = I-a,
$$
where $I$ is the identity map in $\bofh$. The proof is by induction on the length of the word corresponding to $E\in\cB$.

If the length is equal to one then $E\in\cP(\cS_0)$, so that the desired statement follows from {\it Claim~1}.

Suppose that the length is greater than one. Then either $E=E_1\wedge E_2$ or $E=\neg E_1$, for some $E_1$ and $E_2$ in $\cB$ which have length smaller than that of~$E$. By the induction hypothesis, $E_1A=AF_1$ and $E_2A=AF_2$ for some $F_1$ and $F_2\in\cB$. So, in the first case we have: 
$$
EA=(E_1\wedge E_2)A=E_1E_2A=E_1AF_2=AF_1F_2=A(F_1\wedge F_2),
$$
and in the second case
$$
EA=(\neg E_1)A=(I-E_1)A=A(I-F_1)=A(\neg F_1).
$$
This completes the induction. We remark that the projection $F$ obtained by this construction satisfies the equality $F=A^*EA$.

Now, let $E\in\cE(\cS_0)$ be arbitrary. Clearly, it is enough to consider the case when $E$ is the supremum of a collection of projections in~$\cB$. It is easy to see that there exists an increasing net $(E_\alpha)$ of projections in $\cB$ such that $E_\alpha\overset{SOT}{\longrightarrow}E$ (for example, the net $(E_\alpha)$ could be indexed by the finite subcollections of a collection of projections whose supremum is~$E$). By the first part of {\it Claim~2}, fix for each $E_\alpha$ the corresponding projection~$F_\alpha$. Notice that the choice of the projection $F$ in {\it Claim~1} and in the first part of {\it Claim~2}, $F_\alpha=A^*E_\alpha A$, respects the order in $\bofh$: if $E_1\le E_2$ are two projections in $\cP(\cS_0)$ then the correspondent projections $F_1$ and $F_2$ satisfy $F_1\le F_2$. It follows that the net of projections $(F_\alpha)$ is increasing. By \cite[Lemma~I.6.4]{Davidson}, there exists a projection $F\in\cE(\cS_0)$ such that $F_\alpha\overset{SOT}{\longrightarrow}F$. Then
$$
EA=\mbox{{\tiny SOT-}\!}\lim_\alpha(E_\alpha A)=\mbox{{\tiny SOT-}\!}\lim_\alpha(AF_\alpha)=AF,
$$
and {\it Claim~2} is proven.

{\it Claim 3}. The semigroup $\cS_1=\cS(\cS_0,\cE(\cS_0))$ consists of partial isometries. Moreover, $\cP(\cS_1)=\cE(\cS_1)=\cE(\cS_0)$.

For notational convenience, let us assume $I\in\cS_0$. The semigroup $\cS_1$ consists of operators of the form 
$$
T=A_1E_1A_2E_2\cdots A_{n-1}E_{n-1}A_n,
$$
with $A_i\in\cS_0$ and $E_k\in\cE(\cS_0)$. We proceed by induction on~$n$. If $n=1$, then $TT^*=A_1A_1^*\in\cP(\cS_0)$ and $T^*T=A_1^*A_1\in\cP(\cS_0)$, so that $T$ is a partial isometry. Suppose that $n\ge 2$. Let $T_0=A_1E_1A_2E_2\cdots A_{n-1}$. Then
$$
TT^*=(T_0E_{n-1}A_n)(T_0E_{n-1}A_n)^*=T_0E_{n-1}A_nA_n^*E_{n-1}T_0^*.
$$
Let $P=A_nA_n^*$. Then $P\in\cP(\cS_0)\subseteq\cE(\cS_0)$, therefore $PE_{n-1}=E_{n-1}P$. We obtain:
$$
TT^*=T_0E_{n-1}PT_0^*.
$$
Since $T_0$ is a product of operators in $\cS_0$ and $\cE(\cS_0)$, it follows from {\it Claim~2} and the fact that $\cE(\cS_0)$ is commutative that there exist projections $Q$ and $F_{n-1}$ in $\cE(\cS_0)$ such that
$$
PT_0^*=T_0^*Q\quad\mbox{and}\quad E_{n-1}T_0^*=T_0^*F_{n-1}.
$$
Hence
$$
TT^*=T_0T_0^*F_{n-1}Q.
$$
By the induction hypothesis, $T_0T_0^*$ is a projection in $\cE(\cS_0)$. Hence, so is~$TT^*$. The argument for $T^*T$ is similar. 

Since the induction argument above shows that $TT^*\in\cE(\cS_0)$ for all $T\in\cS_1$, we also get the equality $\cP(\cS_1)=\cE(\cS_1)=\cE(\cS_0)$.
\end{proof} 

\section{The structure}\label{section-3}

In this section we obtain the spacial structure of self-adjoint semigroups of partial isometries. The following theorem is stated for irreducible semigroups. We will see that the irreducibility condition is not an essential restriction. We will discuss the reducible semigroups right after the theorem (Remark~\ref{reducible}).

\begin{theorem}\label{integral-representation}
Let $\cS=\cS^*$ be an irreducible semigroup of partial isometries. Then $\hilb=L^2(\Omega,\cK)$, for some probability measure space $\Omega$ and a Hilbert space $\cK$, such that for each $T\in\cS$, the following are true:
\begin{enumerate}
\item the initial space of $T$ is $L^2(X_T,\cK)$ and the final space of $T$ is $L^2(Y_T,\cK)$, for some measurable subsets $X_T$ and $Y_T$ of~$\Omega$;



\item 
$
T=U_{T}\cdot\int_{X_T}^\oplus T_t\,dt,
$
where $T_t\in\cB(\cK)$ is a unitary operator for each $t\in X_T$ and $U_{T}:L^2(X_T,\cK)\to L^2(Y_T,\cK)$ is a surjective isometry defined by  
$$
(U_{T}f)(t)=w_T(t)f(\varphi_T^{-1}(t)),
$$
for some measurability preserving bijection $\varphi_T:X_T\to Y_T$ (modulo subsets of measure zero) and a weight function $w_T:Y_T\to\bbR^+$.
\end{enumerate}

Moreover, the projections in the enveloping band of projections of $\cS$ are exactly the operators of form $P_A(f)(t)=\chi_A(t)f(t)$ ($t\in\Omega$), where $A\subseteq\Omega$ is measurable.
\end{theorem}
\begin{proof}
By Theorem~\ref{env-band}, we may assume that the projections in $\cS$ all commute and form a complete Boolean algebra. Let $\cN$ be the von Neumann algebra generated by $\cP(\cS)$. Clearly, $\cN$ is commutative. By \cite[Theorem~2, p.~236]{Dixmier}, there exists a probability space $(\Omega,\mu)$ such that $\hilb=\int_\Omega^\oplus\hilb_t\,dt$ and $\cN$ consists exactly of ``diagonalizable'' operators with respect to this decomposition. That is, $A\in\cN$ if and only if it is of the form $A=\int_\Omega^\oplus A_t\,dt$, where $A_t\in\cB(\hilb_t)$ is a scalar multiple of identity. We remark that, since $\cP(\cS)$ is a complete Boolean algebra, the set of projections in $\cN$ is equal to~$\cP(\cS)$.

For an arbitrary $T\in\cS$, let $E_T=T^*T$ and $F_T=TT^*$. Clearly, both $E_T$ and $F_T$ belong to~$\cN$. It follows that both $E_T$ and $F_T$ are ``diagonalizable''. Since these two operators are projections, it is easy to see that there exist measurable sets $X_T$ and $Y_T$ such that
$$
E_T=\int_{X_T}^\oplus I_t\,dt
\quad
\mbox{and}
\quad
F_T=\int_{Y_T}^\oplus I_t\,dt,
$$
where $I_t$ stands for the identity operator in $\hilb_t$.

We will construct a measurability preserving bijection $\varphi_T:X_T\to Y_T$ as follows. For an arbitrary measurable set $A\subseteq\Omega$, put
$$
P_A=\int_A^\oplus I_t\,dt.
$$
Let $A\subseteq X_T$ be measurable. The projection $P_A$ belongs to the algebra~$\cN$, and, hence, to the semigroup~$\cS$. Thus, the product $TP_A$ belongs to~$\cS$, too. This product is a partial isometry, with initial space inside $E_T$ and final space inside~$F_T$. It follows that there exists a measurable set $B\subseteq Y_T$, uniquely defined modulo a set of measure zero, such that
$$
(TP_A)(TP_A)^*=\int_{B}^\oplus I_t\,dt.
$$
Define $\phi$ as a set function by $\phi(A)=B$. By repeating the procedure with $T^*$, we can see that $\phi$ defines a bijection between measurable subsets of $X_T$ and measurable subsets of $Y_T$. Moreover, if $A_1\subseteq A_2$, then $\phi(A_1)\subseteq\phi(A_2)$, and $\phi(X_T)=Y_T$. By \cite[Corollary 9.5.2]{Bogachev}, there exists a bijective point map $\varphi_T$ between $X_T$ and $Y_T$, modulo subsets of measure zero, such that $\varphi_T$ is measurability preserving.

Now, let us define a measure $\nu_T$ on $Y_T$ by the formula
$$
\nu_T(B)=\mu(\varphi_T^{-1}(B)),
$$
for all measurable sets $B\subseteq Y_T$. We claim that $\nu_T$ is equaivalent to~$\mu$. Indeed, if $\nu_T(B)\ne 0$ then, by definition, $\mu(\varphi_T^{-1}(B))\ne 0$, so that the space
$$
\cL=\int_{\varphi^{-1}_T(B)}^\oplus\hilb_t\,dt
$$
is non-zero. Clearly, $\cL$ is a subspace of the initial space of~$T$. Therefore, $T|_\cL$ is an isometry. But then
$$
T(\cL)=\int_{B}^\oplus\hilb_t\,dt
$$
is non-zero, implying $\mu(B)\ne 0$. So, $\nu_T\ll\mu$ on~$Y_T$. Repeating the argument with $T^*$, we get: $\nu_T\approx\mu$ on~$Y_T$.

We are now ready to prove the claims of the theorem. First, let us show that each $\hilb_t$ has the same dimension, modulo a set of measure zero. Suppose that this is not true. Let the set $A_n$, where $n\in\bbN\cup\{\infty\}$, be defined by $A_n=\{t\in\Omega\mid\dim\hilb_t=n\}$. Clearly, there must exist two numbers $n> m$ in $\bbN\cup\{\infty\}$ such that $\mu(A_n)\ne 0$ and $\mu(A_m)\ne 0$. Since $\cS$ is irreducible, there exists an operator $T\in\cS$ whose initial space is contained in $P_{A_n}$ and whose final space overlaps non-trivially with $P_{A_m}$ (otherwise, the set $\cS P_{A_n}\hilb$ is perpendicular to the space $P_{A_m}\hilb$, contradicting irreducibility). Since $P_{A_n}$ and $P_{A_m}$ belong to $\cS$, we may assume, replacing $T$ with $P_{A_m}T$ and shrinking $A_n$ and $A_m$, if necessary, that $T^*T=P_{A_n}$ and $TT^*=P_{A_m}$.

Let $\hilb_1$ be the Hilbert space of dimension~$n$, and $\hilb_2$ be the Hilbert space of dimension~$m$. Then the restriction of $T$ to $P_{A_n}\hilb$ is an isometry between $L^2(A_n,\hilb_1)$ and $L^2(A_m,\hilb_2)$. Note that $(A_m,\nu_T)$ is isomorphic to $(A_n,\mu)$ via $\varphi_T$. Since also $\nu_T$ is equivalent to $\mu$ on $A_m$, we conclude that $L^2(A_n,\hilb_1)$ is isometrically isomorphic to $L^2(A_n,\hilb_2)$, via a map which we will denote by~$\widetilde T$. Moreover, if $X\subseteq A_n$ is a measurable subset then the restriction of $\widetilde T$ to $L^2(X,\hilb_1)$ is an isometry onto $L^2(X,\hilb_2)$.

Recall that $n>m$ by the choice of $n$ and~$m$. Let $e_1,\dots,e_{m+1}$ be an orthonormal set of vectors in $\hilb_1$. Define $f_1,\dots,f_{m+1}$ in $L^2(A_n,\hilb_1)$ by
$$
f_k(t)=e_k,\quad k=1,\dots, m+1,\quad t\in A_n.
$$
Let $\tilde f_k=\widetilde Tf_k$, $k=1,\dots,m+1$. Given $\varepsilon>0$, there exist $v_1,\dots,v_{m+1}\in\hilb_2$ and a measurable set $Z\subseteq A_n$ with $\mu(Z)>0$ such that
$$
\norm{v_k-\tilde f_k(t)}\le\varepsilon
$$
for all $t\in Z$ and $k\in\{1,\dots,m+1\}$. Since $\dim\hilb_2=m$, the vectors $\{v_1,\dots,v_{m+1}\}$ are linearly dependent. That is, for some scalars $a_1,\dots,a_{m+1}$, not all zero, we have
$$
\sum_{k=1}^{m+1}a_kv_k=0.
$$
Permuting indices, if necessary, we may assume that
$$
\abs{a_1}=\max_{k=1,\dots,m+1}\abs{a_k}.
$$
Dividing by $a_1$, we get
$$
v_1+\sum_{k=2}^{m+1}b_kv_k=0,
$$
with $\abs{b_k}\le 1$. It follows that 
$$
\bignorm{\tilde f_1(t)+\sum_{k=2}^{m+1}b_k\tilde f_k(t)}\le(m+1)\varepsilon
$$
for all $t\in Z$. Since 
$$
\bignorm{f_1(t)+\sum_{k=2}^{m+1}b_k f_k(t)}\ge 1
$$
for all $t\in Z$ by orthogonality, this contradicts the fact that $\widetilde T|_{L^2(Z,\hilb_1)}$ is a surjective isometry.
%

We have proved that all Hilbert spaces $\hilb_t$ ($t\in\Omega$) have the same dimension. That is,
$$
\hilb=L^2(\Omega,\cK)
$$
for some Hilbert space~$\cK$. Let us now show that, with respect to this structure of~$\hilb$, every operator $T\in\cS$ decomposes as claimed in the statement of the theorem.

Fix arbitrary $T\in\cS$ and let $X_T$, $Y_T$, $\varphi_T$, and $\nu_T$ be defined as in the beginning of the proof. Since $\nu_T\approx\mu$ on $Y_T$, by Radon-Nikodym theorem there exists a measurable function $h_T:Y_T\to\bbR^+$ such that
$$
\mu(B)=\int_B h_T(t)\,d\nu_T(t)
$$
for all measurable $B\subseteq Y_T$. Define $U_{T}:L^2(X_T,\cK)\to L^2(Y_T,\cK)$ by
$$
(U_{T}f)(t)=h^{1/2}(t)f(\varphi^{-1}_T(t)),
$$
for all $t\in Y_t$ and $f\in L^2(X_T,\cK)$. It can be easily checked that $U_{T}$ is a surjective isometry. Also, it is clear that $U_{T}P_{A}\hilb\subseteq P_{\varphi_T(A)}\hilb$ and  $U^{-1}_{T}P_{\varphi_T(A)}\hilb\subseteq P_{A}\hilb$ for every measurable $A\subseteq X_T$.

Consider the operator 
$$
\widehat T=U^{-1}_{T}T=U^{*}_{T}T
$$
(this is well-defined on $L^2(X_T,\cK)$; we let $\widehat T$ be equal to zero on $L^2(\Omega\setminus X_T,\cK)$). Since $T$ maps $P_A\hilb$ to $P_{\varphi_T(A)}\hilb$ and $T^*$ maps $P_{\varphi_T(A)}\hilb$ to $P_A\hilb$ for each measurable $A\subseteq X_T$, the spaces $P_A\hilb=L^2(A,\cK)$ are all $\widehat T$-reducing. It follows that $\widehat T$ commutes with all operators in~$\cN$. By \cite[Theorem~1, p.~187]{Dixmier},
$$
\widehat T=\int_{\Omega}^\oplus T_t\,dt,
$$
where each $T_t\in\cB(\cK)$. Then 
$$
T=U_{T}\cdot\int_{\Omega}^\oplus T_t\,dt.
$$
Note that 
$$
T^*T=\widehat T^*U_{T}^*U_{T}\widehat T=\int_{X_T}^\oplus T^*_tT_t\,dt
$$
is a projection with the range $L^2(X_T,\cK)$. It follows that $T^*_tT_t=I_\cK$ for each $t\in X_T$, so that 
each $T_t$ is a unitary operator on~$\cK$.
\end{proof}


\begin{remark}\label{masa-producing}
If the set of projections in~$\cS$ is masa-generating (that is, the von Neumann algebra generated by this set forms a masa) then our proof shows that the space $\cK$ in the conclusion of Theorem~\ref{integral-representation} is one-dimensional. It follows that in this case, $\cS$ consists of just weighted composition operators.
\end{remark}

%

\begin{remark}\label{reducible}
The only place where irreducibility of the semigroup $\cS$ was essentially used in the proof of Theorem~\ref{integral-representation} is the part showing that all the $\hilb_t$ have the same dimension. For semigroups that are not irreducible, our argument still produces a decomposition of the Hilbert space into a direct integral
$$
\hilb=\int_\Omega^\oplus\hilb_t\,dt,
$$
such that the projections in $\cS$ correspond to measurable subsets of~$\Omega$. Rearranging $\Omega$ and denoting the Hilbert space of dimension $n$ by $\hilb_n$, where $n\in\bbN\cup\{\infty\}$, we get:
$$
\hilb=\bigoplus_{n\in\bbN\cup\{\infty\}}L^2(\Omega_n,\hilb_n),
$$
where some of the $\Omega_n$ may have zero measure. Repeating the argument from the proof of Theorem~\ref{integral-representation} showing that in the irreducible case the dimensions of the correspondent Hilbert spaces are the same, we conclude that each direct summand in the decomposition above is $\cS$-reducing. Moreover, the operators in $\cS$ decompose as
$$
\hilb=\bigoplus_{n\in\bbN\cup\{\infty\}}T_n,
$$
where each $T_n\in L^2(\Omega_n,\hilb_n)$ has the structure as described in the statement of Theorem~\ref{integral-representation}.
\end{remark}

\begin{remark}
The weight functions $w_T(t)$ in the definition of the isometric mapping $U_{T}$ are present to compensate for the fact that the measure space is not ``symmetric'' with respect to the action of the partial isometries in the semigroup: if $T^*T=P_X$ and $TT^*=P_Y$, for measurable sets $X$ and $Y$, then it may happen that $\mu(X)\ne\mu(Y)$. One might hope to define an equivalent measure on the measurable space $(\Omega,\Sigma)$ to correct this. The following example shows that this is not always possible.
\end{remark}

\begin{example}\label{ugly-measure}
Let $\hilb$ be written as a direct sum $\hilb=\oplus_{k\in\bbN}\hilb_k$ of infinite dimensional Hilbert spaces. Define a semigroup $\cS_1$ as the semigroup containing the identity operator and the set of all operators $T\in\bofh$ which are represented as infinite block matrices with respect to this decomposition, with at most one non-zero block, and the non-zero block equal to the identity operator. For example,
$$
T=\begin{bmatrix}
0 & 0 & 0 & \dots \\
I & 0 & 0 & \\
0 & 0 & 0 & \\
& \dots & & 
\end{bmatrix},
$$
belongs to~$\cS_1$.

Analogously, define $\cS_2$ as the semigroup containing the identity operator and the set of all operators 
$$
\begin{bmatrix}
T_{11} & T_{12} & \dots\\
T_{21} & T_{22} & \dots\\
 & \dots
\end{bmatrix},
$$
where at most one block is non-zero, and the non-zero block belongs to~$\cS_1$. We emphasize that in order to put members of $\cS_1$ as blocks of members of~$\cS_2$, we subdivide each summand $\hilb_k$ from the definition of $\cS_1$ into a further direct sum of countably many infinite-dimensional Hilbert spaces. This way, we obtain $\cS_1\subseteq\cS_2$. By induction, we define $\cS_{n}$ to be the operators which are represented as matrices $(T_{ij})$ with at most one non-zero block, and the non-zero block belongs to $\cS_{n-1}$. Again, it is clear that $\cS_{n-1}\subseteq\cS_n$.

Let $\cS_\infty=\cup_{n\in\bbN}\cS_n$. Obviously, $\cS_\infty$ is a self-adjoint semigroup of partial isometries. Finally, let $\cT$ be the semigroup generated by $\cS_\infty$ and its enveloping band of projections, as in Theorem~\ref{env-band}. It is not difficult to show that the obtained semigroup is irreducible.

%
%

Assume that, with respect to some measure space $(\Omega,\mu)$, the operators $U_{T}$ corresponding to operators $T\in\cT$ have the weight functions $w_T$ satisfying $w_T(t)=1$ for all $t$ in the domain of~$\varphi_T$. This implies that, given an operator $T\in\cT$, with the initial space $L^2(X_T,\cK)$ and the final space $L^2(Y_T,\cK)$, we have $\mu(X_T)=\mu(Y_T)$. However, then the structure of $\cT$ implies that every  such $X_T$ must contain an infinite sequence of disjoint subsets of equal measure, forcing $\mu(X_T)=\infty$. Thus, every measurable subset of $\Omega$ has infinite measure.
\end{example}

Note that the measure in Example~\ref{ugly-measure} is atomless. Let us briefly discuss some properties of semigroups with the property that the measure obtained in Theorem~\ref{integral-representation} is atomic.

\begin{proposition}\label{atoms}
Let $\cS=\cS^*$ be an irreducible semigroup of partial isometries. Let $\hilb=L^2(\Omega,\cK)$ be as obtained in Theorem~\ref{integral-representation}. If $(\Omega,\mu)$ has an atom then $(\Omega,\mu)$ is purely atomic.
\end{proposition}
\begin{proof}
By Theorem~\ref{env-band}, we may assume that $\cS$ contains its enveloping band of projections, so that the projection $P_Z$ onto $L^2(Z,\cK)$ belongs to $\cS$ for any measurable subset $Z$ of~$\Omega$.

Let $X\subseteq\Omega$ be an atom in $\Omega$. Suppose $(\Omega,\mu)$ is not purely atomic. Let $\Omega_1$ be the union of all atoms in $(\Omega,\mu)$ and denote the complement of $\Omega_1$ by $\Omega_2$. Then the measure of $\Omega_2$ is not zero and $\Omega_2$ has no atoms. Since $\cS$ is irreducible, there exists an operator $T\in\cS$ such that the initial space of $T$ is equal to $L^2(X,\cK)$, and the final space of $T$ overlaps with $L^2(\Omega_2,\cK)$ (otherwise, the set $\cS L^2(X,\cK)$ is orthogonal to the space $L^2(\Omega_2,\cK)$). Replacing $T$ with $P_{\Omega_2}T$, we may assume that the final space of $T$ is contained in $L^2(\Omega_2,\cK)$. 

Let $Y\subseteq\Omega_2$ be such that $L^2(Y,\cK)$ is the final space of~$\cT$. There exist two disjoint measurable subsets $Y_1$ and $Y_2$ of $Y$ of non-zero measure such that $Y=Y_1\cup Y_2$. The operators $T_1=P_{Y_1}T$ and $T_2=P_{Y_2}T$ both belong to~$\cS$. It follows that the initial spaces of $T_1$ and $T_2$ are disjoint subsets of $X$ of non-zero measure, contradicting the fact that $X$ is an atom.
\end{proof}

\begin{remark}\label{reducible-atomic}
Suppose that the semigroup $\cS$ is reducible. Represent $\hilb$ as a direct integral $\int_{\Omega}^\oplus\hilb_t\,dt$ compatible with the projections in $\cS$, as in Remark~\ref{reducible}. 
Let $X\subseteq\Omega$ be the union of all the atoms in~$\Omega$ and define
$$
\hilb_1=\int_X^\oplus\hilb_t\,dt,
$$
a direct integral over a purely atomic measure, and
$$
\hilb_2=\int_{\Omega\setminus X}^\oplus\hilb_t\,dt,
$$
a direct integral over an atomless measure space. Then, by an argument similar to that of the proof of Proposition~\ref{atoms}, $\hilb_1$ and $\hilb_2$ are $\cS$-reducing.
\end{remark}

The following proposition shows that it is not possible to construct a semigroup like the one described in Example~\ref{ugly-measure} if the underlying measure is purely atomic. Thus we obtain a cleaner version of our main result in this case.

\begin{proposition}\label{measure-correction}
Let $\cS=\cS^*$ be an irreducible semigroup of partial isometries. Let $(\Omega,\mu)$ be a measure space as in Theorem~\ref{integral-representation}. If $\mu$ (has an atom and thus) is purely atomic, then there exists a $\mu$-equivalent measure $\nu$ on $\Omega$ such that each $T\in\cS$ is represented as $T=V_{T}\cdot\int_\Omega^\oplus T_t\,dt$, where $V_{T}$ is given by
$$
(V_{T}f)(t)=f(\varphi_T^{-1}(t)).
$$
\end{proposition}
\begin{proof}
Theorem~\ref{integral-representation} provides a probability measure; it follows that $\Omega$ is the union of countably many atoms. Let $(a_n)$ be some enumeration of the atoms. It is easy to see that, by Theorem~\ref{integral-representation}, every operator has the form $T=V_{T}\cdot\oplus_n T_n$, where $V_{T}$ is a rearrangement of the atoms and the direct summation is over those atoms which are in the initial space of~$T$.

Define the measure $\nu$ by letting
$$
\nu(a_k)=\mu(a_1)
$$
for all~$k$. It is clear that this measure is equivalent to~$\mu$. The formula for $V_{T}$ is also easy to obtain.
\end{proof}

We will now describe some special semigroups whose representations have purely atomic measures. The following notation turns out to be useful with such semigroups.

\begin{notation} 
Let $k\in\bbN\cup\{\infty\}$. We will denote by the symbol $\cS^k_0$ the semigroup of all $k\times k$-matrices having at most one non-zero entry, and that entry equal to one. We will denote by $\cS^k_1$ the semigroup of all $k\times k$ matrices with the property that every row and every column has at most one non-zero entry, and that entry equal to one. 

Let $\cU$ be a semigroup of operators on a (finite- or infinite-dimensional) Hilbert space~$\hilb$. We use the symbols $\cS_0^k(\cU)$ and $\cS_1^k(\cU)$ to denote the semigroups of operators on $\oplus_{i=1}^k\hilb$ whose matrices are obtained from matrices in $\cS_0^k$ (respectively, $\cS_1^k$) by replacing the non-zero entries with elements of $\cU$ and the zero entries with the zero operators on~$\hilb$. If $\cU$ is a group of unitaries then every subsemigroup of a $\cS_1^k(\cU)$ will be referred to as a \term{zero-unitary semigroup}.
\end{notation}

For example, if $\cU$ is the group of unitaries on a Hilbert space~$\hilb$, then a typical member of $\cS_1^4(\cU)$ would look like
$$
\begin{bmatrix}
0 & U_1 & 0 & 0\\
0 & 0 & 0 & 0 \\
U_2 & 0 & 0 & 0 \\
0 & 0 & U_3 & 0
\end{bmatrix},
$$
where $U_1$, $U_2$, and $U_3$ are some unitary operators on~$\hilb$.

\begin{theorem}Let $\cS$ be a self-adjoint finitely generated semigroup of partial isometries. That is, $\cS=\langle T_1,T_1^*,\dots,T_n,T_n^*\rangle$ for some $T_1,\dots,T_n$. If $\ran T_i=\ran T_j$, $\ker T_i=\ker T_j$ and $T_iT_j=T_jT_i$ for all $i$ and~$j$, then $\cS$ is atomic in the sense that the measure space from Theorem~3.1 corresponding to $\cS$ is purely atomic.
\end{theorem}
\begin{proof}
For each $i$, let $\cL_i$ be the lattice of subspaces of $\cH$ generated by the ranges and the kernels of~$T_i^k$, $k=0,1,2,\dots$. Let us show that $\cL_i=\cL_j$ for all $i$ and~$j$. It is enough to establish that $\ran T_i^k=\ran T_j^k$ and $\ker T_i^k=\ker T_j^k$, $k=0,1,2,\dots$. If $k=0$, the statement is trivial. Assume by induction that the statement is proved for some~$k$. Pick $x\in\ran T_i^{k+1}$. There is $y\in\cH$ such that $x=T_i^{k+1}y=T_iT_i^ky$. Since $\ran T_i^k=\ran T_j^k$, there is $z\in\cH$ such that $T_i^ky=T_j^kz$. Then $x=T_iT_j^kz$. By commutativity, $x=T_j^kT_iz$. Since $\ran T_i=\ran T_j$ by the assumptions of the theorem, there is $v\in\cH$ such that $x=T_j^kT_jv=T_j^{k+1}v$. So, $\ran T_i^{k+1}=\ran T_j^{k+1}$.
Similarly, if $x\in\ker T_i^{k+1}$, that is, $T_i^{k+1}x=0$, then $T_i(T_i^kx)=0$. Assuming by induction that $\ker T_i^k=\ker T_j^k$, we get $T_i(T_j^kx=0)$. By commutativity, $T_j^k(T_ix)=0$. Again, using the assumptions of the theorem, $T_j^k(T_jx)=T_j^{k+1}x=0$. So, $\ker T_i^{k+1}=\ker T_j^{k+1}$. This shows that $\cL_i=\cL_j$ for all $i$ and~$j$.

For each $i=1,\dots, n$, let $\cS_i$ denote the singly generated semigroup $\langle T_i,T_i^*\rangle$. By Theorem~2.1, each $\cS_i$ is atomic (in the sense of Theorem~3.1). It readily follows from $\cL_i=\cL_j$ that the atoms corresponding to $\cS_i$ are the same as those for $\cS_j$ ($i,j=1,\dots,n$). So, each generator (and, hence, each member) of $\cS$, when applied to any of these atoms, turns it either into zero or into another atom.
\end{proof}

\begin{remark}
There are examples showing that without the range assumption or the kernel assumption the statement is not true even with commuting generators.
\end{remark}

The next theorem shows that if the semigroup contains a non-zero compact operator then the integral representation for the semigroup is discrete. Moreover, in this case, we do not even need the semigroup to be self-adjoint.

We will need the following standard lemma, whose proof we include for the convenience of the reader.

\begin{lemma}\label{unitary-group}
Suppose that $\cS$ is a norm closed semigroup of matrices and $U\in\cS$ is a unitary. Then $I\in\cS$ and $U^*\in\cS$.
\end{lemma}
\begin{proof}
Note that the unitary matrix $U$ is similar to a diagonal matrix, with every diagonal entry in~$\partial\bbD$. It follows that there is a sequence $(n_i)$ of natural numbers such that $U^{n_i}\to I$. Since $\cS$ is norm closed, $I\in\cS$. Similarly, $U^*=U^{-1}=\lim_{i\to\infty}U^{n_i-1}\in\cS$.
\end{proof}

\begin{theorem}\label{part-isom-semigroups}
Suppose that $\cS$ is an irreducible norm-closed semigroup of partial isometries containing a non-zero compact operator. Then there exists $k\in\bbN\cup\{\infty\}$ and an irreducible group $\cU$ of unitary matrices such that, after a unitary equivalence,
$$
\cS^k_0(\cU)\subseteq \cS\subseteq \cS^k_1(\cU).
$$
The size of matrices in $\cU$ is equal to the minimal non-zero rank of operators in~$\cS$.
\end{theorem}
\begin{proof}
We start by introducing a notation. The symbol $E_{ij}^k$ denotes the $k\times k$ matrix (where $k\in\bbN\cup\{\infty\}$) whose $i,j$-entry is equal to $1$ and whose all other entries are equal to~$0$. If $A$ is an $r\times r$ matrix, then the symbol $E_{ij}^k\otimes A$ denotes the matrix obtained from $E_{ij}^k$ by replacing its $i,j$-entry with $A$ and replacing all other entries with the $r\times r$ zero matrix.

Since a compact partial isometry is necessarily finite-rank, $\cS$ contains finite-rank operators. Denote the minimal non-zero rank of operators in $\cS$ by~$r_0$. Let $\cJ$ be the ideal of operators in $\cS$ of rank at most~$r_0$. Then $\cJ$ is an irreducible semigroup. By Turovskii's theorem (see~\cite{Turovskii}, see also \cite[Theorem~8.1.11]{RR00}), $\cJ$ contains a non-nilpotent operator~$T$. By Theorem~\ref{pwr-prt-isometry}, $T$ is a direct sum of a unitary, forward and backward shifts and their truncations, and a zero. Since $T$ is a finite-rank operator, it cannot contain infinite-dimensional shifts. Also, since $\rk T$ is minimal, it cannot have finite-dimensional truncations of the shift. Therefore, with respect to the decomposition $\hilb = (\ker T)^\perp\oplus\ker T$, we may represent $T$ as
$$
T=\begin{bmatrix}U & 0 \\ 0 & 0\end{bmatrix},
$$
where~$U$ is an $r_0\times r_0$ unitary matrix. By Lemma~\ref{unitary-group}, the operators 
$$
P=\begin{bmatrix}I_{r_0} & 0 \\ 0 & 0\end{bmatrix}
\quad\mbox{and}\quad
T^*=\begin{bmatrix}U^* & 0 \\ 0 & 0\end{bmatrix}
$$
belong to~$\cS$, as well.

Define $\cU=P\cS P|_{P\hilb}\setminus\{0\}$. Since $P\in\cS$, $\cU$ is a semigroup. Also, since $\cS$ is irreducible, so is~$\cU$. Next, if $U\in\cU$ then 
$$
\begin{bmatrix}U & 0 \\ 0 & 0\end{bmatrix}
$$
belongs to $\cS$, and hence it must be a partial isometry. It is clear that this implies that $U$ is a partial isometry. However, the rank of $U$ cannot be smaller than~$r_0$. Thus, $U$ is a unitary. By Lemma~\ref{unitary-group}, $U^*\in\cU$. It follows that $\cU$ is a group of $r_0\times r_0$ unitaries. 

Since $\cS$ is irreducible, there exists a non-zero operator~$B$ such that, for some operators $A$, $C$, and $D$, the operator
$$
S=\begin{bmatrix}A & C \\ B & D\end{bmatrix}\in\cS.
$$
Multiplying by $P$ on the right, we may assume that $C=0$ and $D=0$. Observe that if $A$ were non-zero then it would belong to $\cU$ (just multiply by $P$ on the left), hence it would be unitary. However, the operator $S$ is a partial isometry, which then would force $B$ to be equal to zero. Therefore, $A=0$. Also, since $k$ is the minimal non-zero rank in $\cS$, we must have $\rk B=r_0$. Thus, choosing an appropriate orthonormal basis, we may represent $S$ as
$$
S=\begin{bmatrix}0 & 0 & 0 \\ B & 0 & 0 \\ 0 & 0 & 0\end{bmatrix},
$$
where $B$ is an invertible $r_0\times r_0$ matrix. Again, since $S$ is a partial isometry, we conclude, by considering $S^*S$, that $B^*B=I_{r_0}$, so that $B$ is, in fact, unitary. Choosing an appropriate orthonormal basis, we may assume that $B=I_{r_0}$.

Inductively repeating this argument, we conclude that there is a decomposition of $\hilb$ into a direct sum of $k$ subspaces of dimension~$r_0$, where $k\in\bbN\cup\{\infty\}$, such that the operator $S_{i1}=E_{i1}^k\otimes I_{r_0}$ belongs to $\cS$ for all~$i$.

Let $i\le k$ and let $R$ be an operator in~$\cS$ which has the $(1,i)$-block non-zero (such an $R$ exists by irreducibility of~$\cS$). Multiplying $R$ by $S_{i1}S_{i1}^*$ on the right, we may assume that this is the only non-zero block in~$R$. Denote this block by~$C$. Repeating the argument that we used for the block~$B$ in the operator~$S$ above, we conclude that the $C$ is an $r_0\times r_0$ unitary. Moreover, multiplying $R$ by $S_{i1}$ on the right, we observe that $C\in\cU$. But then, multiplying $S$ by $E_{11}^k\otimes C^{-1}$ on the left, we obtain that $S_{1i}=E_{1i}^k\otimes I_{r_0}\in\cS$.

It follows that each $E_{ij}^k\otimes U=S_{i1}(E_{11}^k\otimes U)S_{1j}\in\cS$ for all $U\in\cU$ and for all $i$ and~$j$. That is, $\cS^k_0(\cU)\subseteq\cS$, and the first part of the theorem is proved. Since each non-zero block of an operator in~$\cS$ must be unitary, and $\cS$ consists of contractions, there may be at most one non-zero block in each block column. Similarly, since the adjoint of each operator in~$\cS$ is also a contraction, each row block can have at most one non-zero block. Hence, $\cS\subseteq\cS_1^k(\cU)$, and the second inclusion in the statement of the theorem is proved, too.
\end{proof}

\begin{corollary} \label{automatic-sa}
Let $\cS$ be an irreducible semigroup of partial isometries containing a compact operator. Then the self-adjoint semigroup generated by $\cS$ consists of partial isometries.
\end{corollary}
\begin{proof}
Let $\overline{\cS}$ be the norm-closure of $\cS$. Then $\overline{\cS}$ consists of partial isometries. It follows from Theorem~\ref{part-isom-semigroups} that there exist $k\in\bbN$ and $\cU$ a group of unitary operators such that $\cS\subseteq\overline{\cS}\subseteq\cS_1(\cU)$. The self-adjoint semigroup generated by $\cS$ consists precisely of finite words whose entries are of form $T$ or $T^*$, where $T\in\cS$. It is easy to see that any such product belongs again to $\cS_1(\cU)$. However, $\cS_1(\cU)$ only contains partial isometries.
\end{proof}

The following corollary, which is an immediate consequence of Theorem~\ref{part-isom-semigroups}, may look surprising.

\begin{corollary}\label{prime-size}
Let $n$ be a prime number and suppose that $\cS$ is an irreducible norm-closed semigroup of partial isometries in $M_n(\bbC)$ having no rank-one members. Then $\cS$ is a group of unitaries. 
\end{corollary}


\begin{thebibliography}{00}
  
  \bibitem{Bogachev}
  V.I.~Bogachev,
  \emph{Measure theory. Vol. I, II.}
  Springer-Verlag, Berlin, 2007.
  
  \bibitem{Davidson}
  K.R.~Davidson,
  \emph{$C^*$-algebras by example}. 
  Fields Institute Monographs, 6. American Mathematical Society, Providence, RI, 1996.

  \bibitem{Dixmier}
  J.~Dixmier,
  \emph{von Neumann algebras}. 
  With a preface by E. C. Lance. Translated from the second French edition by F. Jellett. North-Holland Mathematical Library, 27. North-Holland Publishing Co., Amsterdam-New York, 1981.

  \bibitem{DP85}
  J.~Duncan, A.L.T.~Paterson,
  \emph{$C^*$-algebras of inverse semigroups},
  Proc. Edinburgh Math. Soc. (2) 28 (1985), no. 1, 41--58. 
  

  \bibitem{HW69}
  P.R.~Halmos, L.J.~Wallen,
  \emph{Powers of partial isometries}, J. Math. Mech. 19 1969/1970 657--663.
  
  \bibitem{KR-v2} 
  R.V.~Kadison, J.R.~Ringrose,
  \emph{Fundamentals of the theory of operator algebras. Vol. II.}. Pure and Applied Mathematics, 100. Academic Press, Inc., Orlando, FL, 1986. 

     
  \bibitem{Okninski} 
  J.~Okni\'nski,
  \emph{Semigroups of matrices}. Series in Algebra, 6. World Scientific Publishing Co., Inc., River Edge, NJ, 1998.
    
  \bibitem{Pat98}
  A.L.T.~Paterson, \emph{Groupoids, inverse semigroups, and their operator algebras}, Progress in Mathematics, 170. Birkh\"auser Boston, Inc., Boston, MA, 1999.
  
 \bibitem{RR00} 
  H.~Radjavi, P.~Rosenthal,
  \emph{Simultaneous triangularization}. Springer-Verlag, New York, 2000.
  
  \bibitem{RR08}
  H.~Radjavi, P.~Rosenthal,
  \emph{Limitations on the size of semigroups of matrices},
  Semigroup Forum 76 (2008), no. 1, 25--31.
  
  \bibitem{Turovskii}
  Yu.V.~Turovskii,
  \emph{Volterra semigroups have invariant subspaces}, J. Funct. Anal. 162 (1999), no. 2, 313--322. 
  
\end{thebibliography}
\end{document}